\documentclass{amsart}
\usepackage{amssymb}
\usepackage{amsmath}
\usepackage{amsthm}
\usepackage[latin1]{inputenc}
\usepackage[T1]{fontenc}
\usepackage{hyperref}
\usepackage[all]{xy}
\usepackage{enumerate}
\usepackage{pstricks}
\usepackage{pstricks-add}

\theoremstyle{plain}
\newtheorem{theorem}{Theorem}[section]
\newtheorem{lemma}[theorem]{Lemma}
\newtheorem{proposition}[theorem]{Proposition}

\theoremstyle{definition}
\newtheorem{definition}[theorem]{Definition}

\newtheorem*{notation}{Notation}
\newtheorem{remark}[theorem]{Remark}

\numberwithin{equation}{section}

\newcommand{\Cs}{\ensuremath{\mathrm{C}^*}}
\newcommand{\Csr}{\ensuremath{\mathrm{C}_r^*}}
\newcommand{\E}{\ensuremath{\mathcal{E}}}
\newcommand{\Ew}{\ensuremath{\mathcal{E}^w}}

\DeclareMathOperator{\re}{\R e}

\newcommand{\R}{\mathbb{R}}
\newcommand{\C}{\mathbb{C}}
\renewcommand{\H}{\mathbb{H}}
\newcommand{\GR}{\ensuremath{\mathrm{SL}(2,\R)}}
\newcommand{\GC}{\ensuremath{\mathrm{SL}(2,\C)}}
\newcommand{\GF}{\ensuremath{\mathrm{SL}(2,F)}}

\newcommand{\scal}[2]{\langle #1,#2\rangle}
\newcommand{\ind}[3]{\ensuremath{\mathop{\rm{Ind}}\nolimits_{#1}^{#2}{#3}}}
\newcommand{\mat}[4]{\left[\begin{array}{cc}#1&#2\\#3&#4\end{array}\right]}

\def\bn{\bar{n}}

\def\Rw{\mathcal{R}_w}

\newcommand{\n}{\boldsymbol{\bar{n}}}
\newcommand{\m}{\boldsymbol{m}}
\newcommand{\ab}{\boldsymbol{a}}
\newcommand{\lm}{\boldsymbol{l}}

\newcommand{\triv}{\textbf{1}}

\DeclareMathOperator{\F}{\mathcal{F}}
\DeclareMathOperator{\Fw}{\mathcal{F}_\mathit{w}}
\DeclareMathOperator{\Iw}{\mathcal{I}_\mathit{w}}
\DeclareMathOperator{\Uw}{\mathcal{U}_\mathit{w}}

\newcommand{\N}{\ensuremath{\bar{N}}}

\newcommand{\s}{\mathcal{S}\left(\R^n\right)}
\newcommand{\sF}{\mathcal{S}\left(F^n\right)}
\newcommand{\sdx}{\mathcal{S}\left(F^2\right)}

\begin{document}

\title[$\Cs$-algebraic intertwiners for $\mathrm{SL}(2)$]{$\Cs$-algebraic intertwiners for principal series: case of $\mathrm{SL}(2)$}
\author{Pierre Clare}
\address{Pierre Clare\\ The Pennsylvania State University\\ Department of Mathematics\\ McAllister Building\\ University Park, PA - 16802}
\email{clare@math.psu.edu}
\date{\today}

\setcounter{tocdepth}{3}

\begin{abstract}
We construct and normalise intertwining operators at the level of Hilbert modules describing the principal series of $\GF$ for $F=\R$, $\C$ or $\H$. Normalisation is achieved through the use of a Fourier transform defined on some homogenous space and twisted by a Weyl element. Normalising factors are also explicitely obtained. In the appendix we relate reducibility points to a certain distribution arising from the non-normalised intertwiners.
\end{abstract}

\maketitle


\section{\texorpdfstring{Introduction: a $\Cs$-algebraic point of view on principal series}{Introduction: a C*-algebraic point of view on principal series}}

The study of the tempered dual $\widehat{G}_r$ of a semisimple Lie group $G$ in relation to the Plancherel formula is essentially the work of Harish-Chandra \cite{HC3}, who described the various series of representations carrying the Plancherel measure. Because of its measure-theoretic nature, this approach does not require \emph{all} tempered representations to be taken into account. However, in order to understand $\widehat{G}_r$ as a topological space, a full description of the intertwining relations among the principal series is necessary, for which A. W. Knapp and E. M. Stein developed the main tools in \cite{KS1, KS2}. Facts and references about these matters may be found in \cite{Knapp1} and \cite{Lipsman}.

If $G$ admits more than one conjugacy class of Cartan subgroups, $\widehat{G}_r$ may not be Hausdorff when equipped with the Fell-Jacobson topology. According to the general philosophy of Non-commutative Geometry, the relevant algebra related to $\widehat{G}_r$ is the reduced $\Cs$-algebra of the group, denoted by $\Cs_r(G)$. References on the relations between the properties of $\Csr(G)$ and the representation theory of $G$ are \cite{Dixmier} and \cite{T}.

In order to analyse $\Csr(G)$, it is natural to seek a formulation of the basic objects and results of semisimple theory in terms of operator algebras. Elements of such a $\Cs$-algebraic description of the principal series and Bruhat theory were obtained in \cite{PThese} and \cite{PArtmodules}. The purpose of the present article is to discuss the analogue of Knapp-Stein theory of intertwining operators in this framework for $G=\mathrm{SL}(2)$.

Our main result is Theorem~\ref{thm}, in which an explicit Hilbert module isometry is constructed by means of a Fourier transform defined on a homogenous space of $G$. This operator is considered as a \emph{$\Cs$-algebraic intertwiner} because it commutes to the action of $\Csr(G)$ on its domain, an induction Hilbert module called the \emph{universal principal series} of $G$. Moreover, it relates to the \emph{standard intertwining integral}, an operator defined only on a dense subspace of the universal principal series, in a way that translates the normalising process of Knapp and Stein at the level of Hilbert modules.

The paper is organised as follows. After recalling some structural facts and fixing notations, Section~\ref{UPS} is devoted to the description of the universal principal series $\E$ of $\mathrm{SL}(2)$. In particular, Proposition~\ref{caractE} characterises a convenient subspace of functions in $\E$. Section~\ref{IntertwineSect} starts with a description of the standard intertwining integral. The notion of $\Cs$-normalisation is formally introduced in Paragraph~\ref{Normalprinciple} (Definition~\ref{normaldef}), and the normalisation theorem is established in Paragraph~\ref{thmsect}. In the \hyperref[Appendix]{appendix}, we relate the reducibility points in the classical principal series to a certain distribution responsible for the unboundedness of the standard intertwining integral.

\section{Universal principal series}\label{UPS}

\subsection{Notations}\label{NotaF}

Throughout the paper, $F$ will denote one of the fields $\R$ or $\C$ or the skew field $\H$ of quaternions and $F^\times$ its multiplicative group. Considering $F$ as a real vector space, we let $d_F=\dim_\R(F)$, so that $d_\R=1$, $d_\C=2$ and $d_\H=4$. Denoting $i$, $j$, $k$ the usual quaternionic units, a generic quaternion has the form $q=a+ib+jc+kd$ and a complex (resp. real) number will be considered as a quaternion for which $c=d=0$ (resp. $b=c=d=0$). The \emph{conjugate} of $q$ is defined by $\bar{q}=a-ib-jc-kd$. Finally, we define the \emph{real part} of an element $x$ in $F$ by $\re(x)=\frac{1}{2}\left(x+\bar{x}\right)$, and its \emph{norm} by $|x|_F=\sqrt{\bar{x}x}$, so that $|\cdot|_F$ coincides with the Euclidean norm on $\R^{d_F}$. More generally, $F^n$ will be equipped with $\scal{x}{y}=\sum_{k=1}^n\bar{x}_k y_k$, so that $\re\left(\scal{x}{y}\right)$ coincides with the Euclidean scalar product of $x$ and $y$ seen as vectors in $\R^{d_F n}$. The corresponding vector norm will be denoted by $\|\cdot\|$.

In all that follows, $G_F$ or more simply $G$ will denote the group $\GF$ of square matrices of size 2 with determinant 1. For a general discussion about the determinant of quaternionic matrices, see \cite{AslakenQuat} and \cite{GelfandQuasidet}. In particular, the equality $\left|\det\nolimits_{\R} u\right|=\left|\det\nolimits_{F} u\right|^{d_F}$ holds for $u$ in $\mathrm{M}_n(F)$. As a consequence, the measure on $F^2$ obtained from the Lebesgue measure under the identification $F^2\simeq\R^{2d_F}$ is $G_F$-invariant and satisfies \[\int_{F^2}f(\lambda x)\,dx = \left|\lambda\right|_F^{-2d_F}\int_{F^2}f(x)\,dx = \int_{F^2}f(x\lambda)\,dx\] for any suitable function $f$ and $\lambda$ in $F^\times$.

The proof of our main result essentially relies on the application of a certain Fourier transform on an appropriate space of functions. Following classical texts, we denote by $\s$ the Schwartz space of rapidly decreasing functions over $\R^n$. Using the identification $F^n\simeq\R^{d_F n}$ again, $\sF$ will denote $\mathcal{S}\left(\R^{d_F n}\right)$. Finally, $\delta_x$ denotes the usual Dirac distribution supported at $x$.

We shall use the following definition of the Fourier transform on $F^n$:
\[\F_{F^n} f(\xi)=\int_{F^n}f(x)e^{-2i\pi\re \left(\scal{x}{\xi}\right)}\,dx,\] for any $f\in\sF$. 
The action of the multiplicative group $F^\times$ by dilations will be denoted in the following way: if $f$ is a Schwartz function, then for any $\alpha\in F^\times$, we define $f^\alpha$ by \[f^\alpha(x)=f(x\alpha)\] for $x\in F^n$. Then, the well-known behaviour of the real Fourier transform under dilations generalises to \begin{equation}\label{FourierdilatF}\F_{F^n} \left(f^\alpha\right) = \left|\alpha\right|^{-d_F n}\left(\F_{F^n} f\right)^{\frac{1}{\bar{\alpha}}}.\end{equation}

\subsection{\texorpdfstring{Structural facts for $\GF$}{Structural facts for SL(2,F)}}\label{StructSL2F}	

This paragraph is devoted to a review of the main facts that will be relevant to the study of the representations of $G_F$.

\subsubsection{Borel subgroups and identifications}

Let $\Theta$ be the Cartan involution defined on $G_F$ by $\Theta(g)={^t}\bar{g}^{-1}$. The subgroup $P_F$ of upper triangular matrices is a Borel subgroup of $G_F$. It admits a Langlands decomposition $P_F = M_F A_F N_F$ where \[M_F=\left\{\mat{m}{0}{0}{m^{-1}}\;,\;|m|_F=1\right\},\] \[A_F=\left\{\mat{a}{0}{0}{a^{-1}}\;,\;a>0\right\},\] \[N_F =\left\{\mat{1}{t}{0}{1}\;,\;t\in F\right\}.\] The \emph{$\Theta$-stable Levi component} $L_F=M_F A_F$ of $P_F$ hence identifies with $F^\times$ and its nilradical $N_F$ with $F$, so that $P_F$ is isomorphic to the semi-direct product $F^\times\ltimes F$.

\begin{notation}We will write $d^\times x$ for the Haar measure on the multiplicative group $F^\times$ defined by \[d^\times x=\frac{dx}{|x|_F^{d_F}},\] where $dx$ denotes the Lebesgue measure on $\R^d_F\simeq F$. \end{notation}

\begin{remark}
As the direct product of a compact group and an abelian one, $L$ is amenable. It follows that $\widehat{L}_r$ coincides with $\widehat{L}$ and $\Csr(L)$ with $\Cs(L)$.
\end{remark}

From now on, the subscript ${}_F$ referring to the field will generally be omitted. The nilradical $\bar{N}$ of the opposite Borel subgroup $\bar P=\Theta\left(P\right)$ also identifies with $F$. More precisely, \[n_t= \mat{1}{t}{0}{1} \qquad\text{and}\qquad \bn_t = \mat{1}{0}{t}{1}\] will denote the elements of $N$ and $\bar{N}$ corresponding to $t$ in $F$. Finally, the Weyl group $W$ consists of two elements. The non-trivial one conjugates $N$ and $\bar{N}$, and is represented modulo $M$ by the matrix $w=\mat{0}{-1}{1}{0}$.

\subsubsection{\texorpdfstring{Topology and measure on $G/N$}{Topology and measure on G/N}}

Writing $c(g)$ for the first column of a matrix $g$ in $G$, we notice that $c(g n)=c(g)$ for any $n$ in $N$ and $c(g_1 g_2)=g_1.c(g_2)$ where $g.x$ denotes the matrix multiplication of the vector $x$ in $F^2$ by $g$. In other terms, $c$ induces a $G$-equivariant homeomorphism between $G/N$ and $F^2\setminus\left\{0\right\}$. It will later be useful to remark that the preimage by $c$ of a non-zero vector $x$ in $F^2$ admits \[\left[\begin{array}{c|c}&\\x&w.\dfrac{\bar{x}}{\|x\|^2}\\&\end{array}\right]\] as a representative in $G$. The (unique up to a constant) $G$-invariant measure on $G/N$ (see \cite{WeilInt}, \cite[Ch.VIII]{KnappBeyond}) corresponds under this identification to the restriction of the Lebesgue measure on the $F$-plane to $F^2\setminus\left\{0\right\}$.

\subsection{Principal series}
An important class of representations of $G$ is obtained by \emph{parabolic induction}, as explained below. \begin{notation} The carrying space of a representation $\pi$ will always be denoted by $\mathcal{H}_\pi$.\end{notation} Let $\sigma$ be a unitary irreducible representation of $M$. If $\nu$ is a purely imaginary complex number, we still write $\nu$ for the character of $A$ defined by $a\mapsto e^{\nu\log a}.$ 

\begin{definition}The \textit{principal series} associated to the parabolic subgroup $P$ consists of the representations of the form \[\pi_P^{\sigma,\nu}=\ind{P}{G}{\sigma\otimes\nu\otimes1_N}.\]\end{definition}

This family is parametrised by the dual $\widehat{L}_r=\widehat{M}_r\times\widehat{A}_r$ of $L$. The Weyl group acts on $\widehat{L}_r$ \textit{via} $w.(\sigma,\nu)=\left(\sigma(w^{-1}\cdot w), -\nu\right)$. Results of F. Bruhat~\cite{Bruhat} show that parameters $(\sigma,\nu)$ in the same orbit under $W$ induce equivalent representations and that $\pi_P^{\sigma,\nu}$ is irreducible if $(\sigma,\nu)$ is not a fixed point under $W$: the principal series representations are said to be \textit{generically irreducible}. 

The problem of determining if reducibility indeed occurs at the Weyl-fixed points was dealt with by Knapp and Stein in \cite{KS1}. Candidates for intertwiners appear naturally as integral operators formally satisfying the intertwining relations. However, these operators are given by non-locally integrable kernels. The central object in this theory is the \textit{standard intertwining integral} associated to the element $w$, formally defined by the formula \begin{equation}I_w^{\sigma,\nu}\,f(g) = \int_{\N} f(xw\bn)\,d\bn\label{intw}\end{equation} where $f$ is a function in a dense subspace of $\mathcal{H}_{\pi_P^{\sigma,\nu}}$. Knapp and Stein then proceed in two steps to construct intertwiners. The first consists in making sense of (\ref{intw}) by letting $\nu$ take non-purely imaginary values. The operators $I_w^{\sigma,\nu}$ are then defined as meromorphic functions in the complex variable $\nu$. The second step is called \textit{normalisation} and yields unitary self-intertwiners at the reducibility parameters. More precisely, Knapp and Stein construct complex-valued meromorphic functions $\gamma_\sigma$ such that the operator defined by
\begin{equation}\widetilde{I}_w^{\sigma,\nu}=\frac{1}{\gamma_\sigma(\nu)}I_w^{\sigma,\nu}\label{KSnormal}\end{equation} is unitary for $\nu\in i\R$, provided that $\nu$ is not a pole. When defined, the operators $\widetilde{I}_w^{\sigma,0}$, where $w.\sigma=\sigma$, allow to describe the splitting of $\pi_P^{\sigma,0}$ and the coefficients $\gamma_{\sigma}(\nu)$ are related to the densities in Harish-Chandra's Plancherel formula.

\subsection{\texorpdfstring{The Hilbert modules $\E$ and $\Ew$}{The Hilbert modules E and Ew}}

The theory of induced representations of $\Cs$-algebras originated in M.A. Rieffel's seminal work \cite{RieffelIRC*A}. It relies on the use of bimodules over the $\Cs$-algebras of the ambient and inducing group and contains Mackey's theory of induced representations of locally compact groups, seen at the level of group $\Cs$-algebras, as a special case. One advantage of using Hilbert modules to describe the induction process is for instance the neat expression of Mackey's Imprimitivity Theorem in terms of strong Morita equivalence and crossed product.

However, a direct application of Rieffel's theory in the situation depicted in the previous paragraph fails to enclose all of its specificities. More precisely, it produces a $\Cs(P)$-module, whereas the results of Bruhat discussed above indicate that the relevant parameter space for the principal series is the reduced dual of the $\Theta$-stable Levi component $L$ of $P$, modulo the action of $W$. This observation suggests that principal series should be induced by a Hilbert module over $\Cs(L)$. Following this idea, a slight generalisation of Rieffel's construction was obtained in \cite{PThese,PArtmodules}, a special case of which leads to a $\Cs(L)$-module $\mathcal{E}(G/N)$ providing an accurate description of the principal series induced from $P$. We recall here the properties of this object that will be of use in the case at hand. Proofs and details can be found in \cite{PArtmodules}.

\subsubsection{\texorpdfstring{Construction in the case of $\GF$}{Construction in the case of SL(2,F)}}

The space of continuous functions with compact support on a topological space $X$ will be denoted by $C_c(X)$. The Hilbert module $\E=\E(G/N)$ is obtained in general as the completion of $C_c(G/N)$ with respect to certain inner product taking values in $C_c(L)\subset \Cs(L)$. Here, according to the discussion of Paragraph~\ref{NotaF}, it will be constructed from functions on $F^2\setminus\left\{0\right\}$.

Let $f,g$ be functions on $F^2$ and $\alpha$ an element of $F$. We define \begin{equation}\scal{f}{g}(\alpha) = |\alpha|^{d_F}\int_{F^2}\overline{f(x)}g(x \alpha)\,dx\label{scal}\end{equation} whenever it makes sense. Proposition~\ref{caractE} below will show that functions in $\sdx$ define elements in $\E$.

Further, we let $L$ act by
\begin{equation}
f.l = |l|^{-d_F}f^{l^{-1}}\label{actL}
\end{equation}
for $l\in F^\times\simeq L$. This formula integrates to \begin{equation}\left(f.\varphi\right)(x)=\int_{F^\times}f(xl^{-1})\varphi(l)|l|^{-d_F}\,d^\times l\label{actCcL}\end{equation} for $f$ suitable, $x$ in $F^2$ and $\varphi$ in $C_c(L)\subset\Cs(L)$.

\begin{definition}\label{GC*G}
The Hilbert module $\E$ over $\Cs(L)$ is obtained from $C_c(F^2\setminus\left\{0\right\})$ by extending the action (\ref{actCcL}) to $\Cs(L)$, then completing with respect to the norm induced by (\ref{scal}). 
\end{definition}

As a reflection of the influence of the Weyl group, $\E$ carries another Hilbert module structure over $\Cs(L)$. More precisely, the non-trivial Weyl element $w$ acts by conjugation on $C_c(L)$: for $\varphi\in C_c(L)$, we denote $\varphi^w:l\mapsto\varphi(w^{-1}lw)=\varphi(l^{-1})$. This action extends to an involutive automorphism $a\mapsto a^{w}$ of $\Cs(L)$, which induces a new $\Cs(L)$-Hilbert module structure on $\E$ in the following way: for $\xi,\eta\in\E$ and $a\in \Cs(L)$, we define

\begin{equation}\xi._{w}a = \xi.a^w\end{equation}

\begin{equation}\scal{\xi}{\eta}_w = \left(\scal{\xi}{\eta}\right)^w.\end{equation}

\begin{notation} The module thus obtained will be denoted by $\Ew$.\end{notation}

The module $\Ew$ can also be introduced in the same way as $\E$ in Definition~\ref{GC*G} by replacing (\ref{scal}), (\ref{actL}) and (\ref{actCcL}) respectively by

\begin{equation}\scal{f}{g}_w(l) = \scal{f}{g}(l^{-1}) = |l|^{-d_F}\int_{F^2}\overline{f(x)}g(l^{-1} x)\,dx,\label{scalw}\tag{\ref{scal} $w$}\end{equation}
\begin{equation}f._{w}l = |l|^{d_F}f^l,\label{actLw}\tag{\ref{actL} $w$}\end{equation}
\begin{equation}\left(f._{w}\varphi\right)(x)=\int_{F^\times}f(xl^{-1})\varphi(l)|l|^{d_F}\,d^\times l.\label{actCcLw}\tag{\ref{actCcL} $w$}\end{equation}

Beside the right $\Cs(L)$-module structure, both $\E$ and $\Ew$ carry an action by linear bounded (in fact compact) operators of $\Cs(G)$, that factorises through $\Csr(G)$. This action is first defined at	the level of $G$ by
\begin{equation}\label{actG}g.f = f\left(g^{-1}\cdot\right),\end{equation}
then integrated to compactly supported functions on $G$ and finally extended to $\Cs(G)$.

\begin{remark}An interesting problem is to eventually characterise the image of the morphism \[\Csr(G)\longrightarrow\mathcal{K}_{\Cs(L)}\left(\E(G/N)\right),\] for it is expected to be complemented and help understand the structure of $\Csr(G)$ in relation with the representation theory of $G$, as advocated in \cite{PArtmodules}.\end{remark}

\subsubsection{\texorpdfstring{Characterisation of functions in $\E$}{Characterisation of functions in E}}

In what follows, it will be convenient to realise Schwartz functions on $F^2$ as elements in $\E$ and $\Ew$.

\begin{proposition}\label{caractE}
The inclusion maps \[\iota:C_c^\infty\left(F^2\setminus\left\{0\right\}\right)\longrightarrow\E\qquad\text{and}\qquad\iota_w:C_c^\infty\left(F^2\setminus\left\{0\right\}\right)\longrightarrow\Ew\] extend by continuity to $\sdx$.
\end{proposition}

\begin{proof}
Let $f$ be in $\sdx$ and consider the family of truncations $\left\{f_n\right\}_{n\geq1}$ defined by $f_n=f.\chi_n$, where for any integer $n$, the function $\chi_n$ is chosen on $F^2\setminus\left\{0\right\}$ to be \begin{itemize}\item smooth,\item compactly supported,\item vanishing on a neighbourhood of the origin,\item constantly equal to 1 on the annulus $\left\{\frac{1}{n}\leq\|x\|\leq n\right\}$.\end{itemize} 
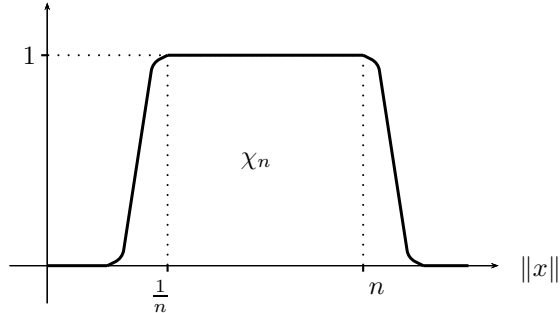
\begin{center}
\begin{pspicture}(-.5,-.5)(6.5,3.9)
\psaxes[ticks=none, linewidth=0.01, labels=none]{->}(0,0)(-0.5,-0.5)(6,3.5)
\psline[linewidth=0.04](0,0)(0.8,0)
\psline[linewidth=0.04](5,0)(5.6,0)
\psline[linewidth=0.04](1.6,2.8)(4.2,2.8)
\psline[linewidth=0.04,linearc=0.21](0.8,0)(0.9,0.05)(1,0.1)(1.4,2.7)(1.5,2.75)(1.6,2.8)
\psline[linewidth=0.04,linearc=0.20](4.2,2.8)(4.3,2.75)(4.4,2.7)(4.8,0.1)(4.9,0.05)(5,0)
\psline[linestyle=dotted](1.6,0)(1.6,2.8)
\psline[linestyle=dotted](4.2,0)(4.2,2.8)
\psline[linestyle=dotted](0,2.8)(4.2,2.8)
\psline(1.6,0)(1.6,-.08)
\psline(4.2,0)(4.2,-.08)
\psline(0,2.8)(-0.08,2.8)
\uput[0](1.2,-0.4){$\frac{1}{n}$}
\uput[0](4.1,-0.3){$n$}
\uput[0](-0.5,2.8){$1$}
\uput[0](6.1,-0.05){$\|x\|$}
\uput[0](2.4,1.4){$\chi_n$}
\end{pspicture}
\end{center}

For any $n$, the truncated function $f_n$ is compactly supported on $F^2\setminus\left\{0\right\}$. We will prove that $\scal{f-f_n}{f-f_n}$ converges to 0 with respect to the norm of $\Cs(L)$, so that $\left\{f_n\right\}_{n\geq1}$ is a Cauchy sequence in $\E$, with limit $f$.

Let us first observe that a straightforward change of variables implies that \begin{equation}\label{scalcsr}\left|\scal{h}{h}(l)\right|\leq\left\|h\right\|_\infty\left\|h\right\|_1\mathrm{min}\left(\left|l\right|^{d_f},\left|l\right|^{-d_F}\right)\end{equation}
for any $h$ in $\sdx$ and $l$ in $F^\times$. It follows that $\varphi_n=\scal{f-f_n}{f-f_n}$ is integrable on $L$ hence belongs to $\Cs(L)$. To prove that $\varphi_n$ converges to $0$ in $\Cs(L)$, it is enough to prove that $\left\|\varphi_n\right\|_1$ converges $0$, which is a consequence of (\ref{scalcsr}) and the definition of $f_n$:\[\left\|\varphi_n\right\|_1\leq C\left\|f-f_n\right\|_\infty\left\|f-f_n\right\|_1\leq C'\frac{\left\|f\right\|_\infty^2}{n},\] where $C$ and $C'$ are constants, which concludes the proof. The statement relative to $\Ew$ is proved along the exact same lines using formulas (\ref{scalw}) and (\ref{actLw}).\end{proof}

\subsubsection{Motivation}

The first feature of Rieffel's modules consists in the way they implement the induction functor by taking tensor products. In a more Lie-theoretic context, $\E$ plays the same role with respect to \emph{parabolic induction}: as proved in \cite{PArtmodules}, for any $\sigma,\nu\in\widehat{L}_r$, there is a unitary map \[\E\otimes_{\Cs(L)}\mathcal{H}_{\sigma\otimes\nu}\stackrel{\sim}{\longrightarrow}\mathcal{H}_{\pi_P^{\sigma,\nu}}\] that intertwines the left actions of $\Csr(G)$ on both Hilbert spaces. In this sense, the Hilbert module $\E$ is considered to globally enclose the principal series representations, hence the name \emph{universal principal series} sometimes used for an analogous object in the $p$-adic situation.

As an example of the way in which classical properties of principal series reflect at the level of Hilbert modules, it was proved in \cite{PArtmodules} that the commutant of $\Csr(G)$ in $\mathcal{L}_{\Cs(L)}\left(\E\right)$ reduces to the center of the multiplier algebra of $\Cs(L)$, acting by right multiplications. These operators correspond to the homotheties of the module $\E$, so that the result may be interpreted as a generic irreducibility statement in this global context.

The next step in the theory, and the purpose of the following section, consists in constructing analogues of Knapp-Stein intertwiners at the level of Hilbert modules.

\section{Intertwining operators}\label{IntertwineSect}

\subsection{Standard intertwining integrals}

The starting point of Knapp and Stein's theory of intertwining operators is the observation that, although it does not converge, the integral formula~(\ref{intw}) formally satisfies intertwining relations. The point of what follows is to study this integral on a subspace of $\E$ and \emph{normalise} it in an appropriate sense in the Hilbert module context, to be made precise in Paragraph~\ref{Normalprinciple}.

The first step was obtained in \cite{PArtmodules}, where the convergence of \begin{equation}\Iw f(g) = \int_{\N} f(xw\bn)\,d\bn\label{Iw}\end{equation} is established for any compactly supported continuous function $f$ on $G/N$. This integral is the analogue of the standard operator $I_w^{\sigma,\nu}$ in Knapp-Stein theory. However, although $\Iw$ defines a map $C_c(G/N)\longrightarrow C(G/N)$, it does not extend to $\E(G/N)$, the obstruction being concentrated in a certain distribution $T_w$, discussed in the appendix.

To conclude this paragraph, let us write a concrete expression for the operator $\Iw$ in the case of $\GF$. With the notations of Paragraph~\ref{StructSL2F}, for $x\in F^2\setminus\left\{0\right\}$ and $t\in\R$, \[\left[\begin{array}{c|c}&\\x&w.\dfrac{\bar{x}}{\|x\|^2}\\&\end{array}\right] w\bn_t = \left[\begin{array}{c|c}&\\-xt + w.\dfrac{\bar{x}}{\|x\|^2}&-x\\&\end{array}\right]\]
so that (\ref{Iw}) becomes \begin{equation}\Iw f(x) = \int_F f\left(tx + w.\frac{\bar{x}}{\|x\|^2}\right)\,dt.\label{IwSL2R}\end{equation}

\subsection{\texorpdfstring{$\Cs$-algebraic normalisation: principle}{C*-algebraic normalisation: principle}}\label{Normalprinciple}

In view of the classical theory of Knapp and Stein, the normalisation process should associate to $\Iw$ a unitary operator of Hilbert modules, satisfying some linearity properties with respect to the left and right actions of $\Csr(G)$ and $\Cs(L)$ respectively.

Regarding $\Cs(L)$-linearity, an easy computation shows that the operator \[\Iw:C_c(G/N)\longrightarrow C(G/N)\] satisfies \[\Iw(f.\varphi)=\Iw(f).\varphi^w=\Iw(f)._{w}\varphi\] for $\varphi\in C_c(L)$, notations (\ref{actCcL}) and (\ref{actCcLw}) still making sense for $\Iw(f)$ in $C(G/N)$.

Finally, considering the normalisation relation (\ref{KSnormal}) in the classical case and taking into account the previous discussion, one is lead to require the following properties for an operator to be considered as a normalisation of $\Iw$.

\begin{definition}\label{normaldef}
An operator $\Uw$ shall be said to \textit{normalise} $\Iw$ if:
\begin{enumerate}[($i$)]
\item $\Uw$ is an isometry of Hilbert modules over $\Cs(L)$ between $\Ew$ and $\E$, that commutes to the left action of $\Csr(G)$;
\item there exists a continuous function $\gamma:L\longrightarrow\C$ such that the composition $\Iw\circ\Uw$ acts on a dense subspace of functions in $\Ew$ by right convolution with $\gamma$ over $L$.
\end{enumerate}
\end{definition}

The situation is illustrated by the following commutative diagram, in which the dense subspace of functions mentioned in the above definition is denoted by $\Ew_0$, and $\E_0$ is a subspace of $\E$ to which $\Iw$ extends in order to make the composition $\Iw\circ\Uw$ well-defined.	


\[\xymatrix @R=2mm @C=15mm {C_c(G/N)\ar[r]^-{\Iw}&C(G/N)\\\cap&\\ \E_0\ar@{-->}[uur]^{\Iw}&\Ew_0\ar[l]^{\Uw\left|_{\Ew_0}\right.}\ar[uu]_{c_\gamma}\\\cap&\cap\\\E&\Ew\ar@<0.5ex>@{-->}[l]^{\Uw}}\]

\subsection{Main result}\label{thmsect}

This paragraph is devoted to an explicit construction in the case of $\GF$ of an operator $\Uw$ together with a function $\gamma$ satisfying the properties of Definition~\ref{normaldef}. The proof will show that the space $\sdx$, seen as a submodule of $\E$ and $\Ew$ according to Proposition~\ref{caractE}, can be used as $\E_0$ and $\Ew_0$. Let $\Fw$ be the operator defined on $\mathcal{S}(F^2)$ by \begin{equation}\Fw f(x)=\int_{F^2}e^{-2i\pi\re\left(\scal{w.y}{\bar{x}}\right)}f(y)\,dy = \F_{F^2} f(w^{-1}.\bar{x}).\label{Fw}\end{equation}

\begin{theorem}\label{thm}
The operator $\Fw$ extends to an operator $\Uw$ on $\E^w$ such that:
\begin{enumerate}[(1)]
\item $\Uw$ normalises $\Iw$,
\item the corresponding function $\gamma$ is given on $L\simeq F^\times$ by \[\gamma(l)=e^{-2i\pi\re \left(l^{-1}\right)}.\]
\end{enumerate}
\end{theorem}

\begin{proof}
In view of formulas~(\ref{actCcL}) and (\ref{actCcLw}), linearity with respect to $\Cs(L)$ reduces to equivariance with respect to $L\simeq F^\times$, which in turn follows directly from the definitions (\ref{actL}) and (\ref{actLw}) together with the property (\ref{FourierdilatF}) of the Fourier transform. In the same way, $\Csr(G)$-linearity can be tested using (\ref{actG}). More precisely, a straightforward computation leads to \[\Fw (g.f)=(w\;{}^tg\;w^{-1}).\Fw(f)\] for $g$ in $G$ and $f$ in $\sdx$. Since the relation $wgw^{-1} = {}^t g^{-1}$ holds for any $g$ in $\GF$, it follows that $\Fw$ commutes to the action of $\Csr(G)$.

Let $f\in\sdx$. Another consequence of (\ref{FourierdilatF}) is the following relation: \begin{equation}\label{Fwdilat}\left(\Fw f\right)^l(x)=|l|^{-2d_F}\Fw\left(f\right)^{l^{-1}}(x).\end{equation} Then, given a fixed $l\in F^\times$,
\begin{align*}
\scal{\Fw f}{\Fw f}(l) & = |l|^{d_F}\scal{\Fw f}{\left(\Fw f\right)^l}_{L^2}\\
& = |l|^{d_F}|l|^{-2d_F}\scal{\F_{F^2} f}{\F_{F^2}\left(f^{l^{-1}}\right)}_{L^2} & \text{by (\ref{Fwdilat})}\\
& = |l|^{-d_F}\scal{f}{f^{l^{-1}}}_{L^2}&\text{by the Plancherel formula on $\R^{2d_F}$}\\
& = \scal{f}{f}_w(l).
\end{align*}

Starting with a function $f$ in $\sdx$, which is stable under Fourier transform, Proposition \ref{caractE}, together with the above computation proves that $\Fw$ maps $\Ew$ to $\E$ isometrically.

Let us now compute the composed map $\Iw\circ\Fw$. It follows from (\ref{IwSL2R}) that $\Iw$ is defined by the kernel

\begin{equation}K_{\mathcal{I}}(x,z) = \int_F\delta_0\left(z-tx - w.\frac{\bar{x}}{\left\|x\right\|^2}\right)\,dt.\label{keriw}\end{equation}

On the other hand, (\ref{Fw}) implies that the kernel $K_w$ defining $\Fw$ is \begin{equation}K_w(z,y) = e^{-2i\pi\re\left(\scal{w.y}{\bar{z}}\right)}.\label{kerfw}\end{equation}

In view of formulas (\ref{keriw}) and (\ref{kerfw}), the composition $\Iw\circ\Fw$ is defined by the kernel $K$: \[\Iw\circ\Fw f(x)=\int_{F^2} K(x,y)f(y)\,dy\] where \[K(x,y)=\int_{F^2}K_{\mathcal{I}}(x,z)K_w(z,y)\,dz,\] that is \[K(x,y) = \int_{F^2}\int_F e^{-2i\pi\re\left(\scal{w.y}{z}\right)}\delta_0\left(z-tx - w.\frac{\bar{x}}{\left\|x\right\|^2}\right)\,dt\,dz.\]

For any vector $x\neq0$ of the $F$-plane, we denote by $\widetilde{x}$ its projection on the Euclidean sphere, so that $x = \widetilde{x}\|x\|$. Then,

\begin{eqnarray*}
K(x,y) & = & \int_F e^{-2i\pi\re\left(\scal{w.y}{\bar{t}\bar{x} + w.\frac{x}{\left\|x\right\|^2}}\right)}\,dt\\
& = & e^{-2i\pi\re\left(\scal{w.y}{w.x}\right)\|x\|^{-2}}\int_F e^{-2i\pi\re \left(\scal{w.y}{\bar{x}}\bar{t}\right)}\,dt\\
& = & \dfrac{e^{-2i\pi\re\left(\scal{y}{x}\right)\|x\|^{-2}}}{\|x\|^{d_F}\|y\|^{d_F}}\F_F\left(\boldsymbol{1}_F\right)\left(\scal{w.\widetilde{y}}{\widetilde{\bar{x}}} \right)
\end{eqnarray*}
so that \[K(x,y) = \frac{e^{-2i\pi\re\left(\scal{y}{x}\right)\|x\|^{-2}}}{\|x\|^{d_F}\|y\|^{d_F}}\delta_0\left(\scal{w.\widetilde{y}}{\widetilde{\bar{x}}}\right).\]

The distribution obtained by fixing a vector $x$ in the kernel $\delta_0\left(\scal{w.\widetilde{y}}{\widetilde{\bar{x}}}\right)$ sends a Schwartz function $\varphi$ to its integral on the subspace $\bar{x}^\perp$ orthogonal to $\bar{x}$: \[\delta_0\left(\scal{\cdot}{\widetilde{x}}\right)\varphi = \int_{\bar{x}^\perp}\varphi.\]


Since $\bar{x}^\perp$ is the $F$-line generated by $x$, one has: \begin{eqnarray*}
\Iw\circ\Fw f(x) & = & \|x\|^{-d_F}\int_{F.x}e^{-2i\pi\re\left(\scal{y}{x}\right)\|x\|^{-2}}f(y)\frac{dy}{\|y\|^{d_F}}\\
& = & \|x\|^{-d_F}\int_{F}e^{-2i\pi\re\left(\scal{\lambda\widetilde{x}}{x}\right)\|x\|^{-2}}f(\lambda\widetilde{x})\frac{d\lambda}{|\lambda|^{d_F}}\\
& = & \|x\|^{-d_F}\int_{F}e^{-2i\pi\re\left(\lambda\right)\|x\|^{-1}}f(\lambda\widetilde{x})\frac{d\lambda}{|\lambda|^{d_F}}\\
& = & \int_{F^\times}e^{-2i\pi\re\left(\lambda\right)}f(\lambda x)\,d^\times\lambda,
\end{eqnarray*}
which proves that $\Iw\circ\Fw f$ is equal to $f*_{F^\times}\gamma$, where $\gamma$ is the expected function.\end{proof}

A normalisation result has already been obtained for $\mathrm{SL}(2,\R)$ and $\mathrm{SL}(2,\C)$ in the author's thesis \cite{PThese} by means of functionnal calculus and differential operators. However, besides including the quaternionic case, it seems to us that the point of view adopted in this article is more natural, for it relies on a more geometric interpretation of the space $G/N$. The approach to Knapp-Stein operators as geometric integral transforms seems to originate in the work of A. Unterberger (see \cite{Unterbergerlivre,PevznerUnter}) in a context of geometric analysis. It has been promoted in various recent works in representation theory, especially in the case of degenerate principal series. For instance, they were obtained \textit{via} Radon transforms on light cones in \cite{KobayashiMano_Schrodinger}, understood in terms of the so-called $\mathrm{Cos}^\lambda$ transformations in \cite{PasquOlaf} and normalised by symplectic Fourier transforms in \cite{Kobayashi_small_GL} and \cite{PArtSpnC}.

\section{Appendix: residual intertwining distribution and reducibility}\label{Appendix}

The purpose of this rather independant part is to suggest that the information about reducibility points in the principal series contained in the normalising factors of the classical Knapp-Stein operators can be extracted from the non-normalised operators $\Iw$. We only consider the cases of $F=\R$ and $F=\C$, in which the principal series are particularly easily described, although distinct phenomena regarding reducibility already occur.

\subsection{Structure and notation}

We retain the notations of Paragraph~\ref{StructSL2F}, so that $M_\R\simeq\left\{\pm1\right\}$ and $M_\C\simeq\mathrm{U}(1)$. The properties of the Bruhat decomposition imply that $G\setminus \N MAN$ has Haar measure $0$. Almost every element $g\in G$ thus admits a unique decomposition \[g=\n(g)\m(g)\ab(g)n_g=\n(g)\lm(g)n_g\] according to $\N MAN$, which is called \emph{the open Bruhat cell}. See \cite[p.513]{KS1} for general properties of the projections $\n$, $\ab$ and $\lm$. In the case at hand, one has for $g=\mat{a}{b}{c}{d}$ such that $a\neq0$,


\[\begin{array}{lclcl}\m(g) = \dfrac{a}{|a|}&\quad&\ab(g) = |a|&\quad&\n(g) = \dfrac{c}{a}\\\lm(g) = a&&&&n_g = \dfrac{b}{a}.\end{array}\]

\subsection{Open picture and truncation}

The existence of an open Bruhat cell manifests at the Hilbert module level through an isometry \begin{equation}\label{openpict}\E\simeq L^2(\N)\otimes \Cs(L)\end{equation} of $\Cs(L)$-modules, referred to as the \emph{open picture} of $\E$. This isomorphism allows to transport the left $\Cs(G)$-structure on $\E$ naturally coming from the action $G\curvearrowright G/N$, to $L^2(\N)\otimes \Cs(L)$ although $G$ does not act on $\N MA$. A practical consequence of (\ref{openpict}) is that $C_c(\N)\otimes C_c(L)$ may be used as a dense subset of $\E$. The correspondence between this isomorphism and the classical non-compact picture of the principal series is discussed in \cite{PArtmodules}.

Whether in the classical context or in the $\Cs$-algebraic framework, one of the advantages of working in the open picture is that it essentially reduces to some analysis on $\N$, where the action of $A$ by dilations can be quantitatively measured. More precisely, there exists a \textit{norm function} $|\cdot|$ on $\N\setminus\left\{1\right\}$ that satisfies homogeneity properties under dilations by the elements of $A$. In the present case, this function coincides with the usual modulus function on $\R$ or $\C$.

An expression for $\Iw$ in the open picture was obtained in \cite{PThese}. For an elementary tensor $f\otimes\varphi$ and an element $x_0=\bn_0 l_0=\bn_0 m_0 a_0\in\N L$, it is given by

\[\Iw (f\otimes\varphi)(x_0) = \int_{\overline{N}}f(\bn_0\bn)\left[U_{\lm(w^{-1}\bn)}\varphi\right]^w(l_0)\,\frac{d\bn
}{|\bn|}\] where $U_\gamma$ denotes the multiplier associated to an element $\gamma$ in a group $\Gamma$ acting on $\Cs(\Gamma)$ by the formula $U_\gamma.\varphi=\varphi(\gamma^{-1}\cdot)$ for $\varphi\in C_c(\Gamma)$, and $\varphi^w=\varphi\circ c_w$ as previously, $c_w$ denoting the conjugation by $w$.

Let us consider the following decomposition of the integral: \[\mathcal{I}_w = \mathcal{I}_w^0+\mathcal{I}_w^{\infty}+\Rw\] with

\begin{gather*}
\mathcal{I}_w^{\infty}(f\otimes\varphi)(x_0) = \int_{|\bn|>1}f(\bn_0\bn)\left[U_{\lm(w^{-1}\bn)}\varphi\right]^w(l_0)\,\frac{d\bn
}{|\bn|}\\
\mathcal{I}_w^0(f\otimes\varphi)(x_0) = \int_{|\bn|<1}\left(f(\bn_0\bn)-f(\bn_0)\right)\left[U_{\lm(w^{-1}\bn)}\varphi\right]^w(l_0)\,\frac{d\bn
}{|\bn|}\\
\Rw(f\otimes\varphi)(x_0) = f(\bn_0)\int_{|\bn|<1}\left[U_{\lm(w^{-1}\bn)}\varphi\right]^w(l_0)\,\frac{d\bn
}{|\bn|}\label{Rw}\tag{\dag}
\end{gather*}

This truncation will allow to extract what prevents $\Iw$ to define a Hilbert module operator, precisely enough to recover the information that is encoded in the singularities of the Knapp-Stein operators, as the main result of this note shows.

It is proved in \cite{PThese} that $\mathcal{I}_w^{\infty}$ and $\mathcal{I}_w^0$ give densely defined operators between the Hilbert modules $\mathcal{E}$ and $\Ew$. The remaining term $\Rw$ hence concentrates the singularity of the global intertwining integral.

\subsection{Residual distribution} 

The first remark about $\Rw$ is that it acts non-trivially only on the right factor $\Cs(L)$ in the open picture of $\mathcal{E}$. In other terms, it is supported in $L$ seen in $G/N\stackrel{.}{=}\N L$. The following results show that $\Rw$ is characterised by a certain distribution on $L$.

Let us first recall the following lemma, proved in \cite[p.97]{PThese}. We sketch the proof below for the reader's convenience.

\begin{lemma}
Let $T_{w,\bn}$ be the distribution defined on $L$ by \[T_{w,\bn}(\lambda)=\delta_1(c_{w^{-1}}\left(\lm(w^{-1}\bn)^{-1}\right)\lambda).\]
The integral \[T_w=\int_{|\bn|<1}T_{w,\bn}\,\frac{d\bn}{|\bn|}\] then defines a Radon measure on $L$.
\end{lemma}

\begin{proof}
Consider $\varphi=\varphi_M\otimes\varphi_A\in C_c^\infty(M)\otimes C_c^\infty(A)$. By definition of $T_{w,\bn}$,
\begin{eqnarray*}
T_{w,\bn}(\varphi)&=&\varphi_M\left(c_w(\m(w^{-1}\bn))\right)\,\varphi_A\left(c_w(\ab(w^{-1}\bn))\right)\\
&=&\varphi_M\left(c_w(\m(w^{-1}\bn))\right)\,\varphi_A\left(\ab(w^{-1}\bn)^{-1}\right)
\end{eqnarray*}
Elementary properties of the projection $\m$ imply that the first factor is homogeneous of degree $0$ under dilations, while the second one only depends on $|\bn|$. More precisely, $\varphi_A\left(\ab(w^{-1}\bn)^{-1}\right)=\varphi_A(|\bn|^{c})$ for some positive constant $c$ depending on the identification between the Lie algebra of $A$ and $\R$.
One can then apply the integration formula of \cite[Proposition 3, p.496]{KS1}, which gives \[T_w(\varphi)=C.M_w(\varphi_M)\int_{0}^{+\infty}\varphi_A(r)\,\frac{dr}{r}\]
where $C$ is a constant and $M_w(\varphi_M)$ denotes the mean value over the unit sphere defined relative to $|\bn|$ of the function $\bn\mapsto\varphi_M\left(c_w(\m(w^{-1}\bn))\right)$. A straightforward estimate relying on the same integral formula proves that this mean value defines a Radon measure on $M$. The distribution $T$ acts on $\varphi_A$ like the Haar measure, hence the result.
\end{proof}

Plugging the formula defining $T_w$ into the expression (\ref{Rw}) of the residual part of the global intertwining integral, one immediately obtains the following characterisation of $\Rw$ as a convolution operator by this distribution:

\begin{proposition}
For any $f\otimes\varphi\in L^2(\N)\otimes C_c(L)$, \[\Rw(f\otimes\varphi)=f\otimes\left(T_w*\varphi^w\right).\]
\end{proposition}

From now on, $T_w$ will be called \textit{the residual distribution} associated to $\Iw$. Let us finally give an explicit expression for $T_w$. Using our identifications, one observes for $\bn_t\in\N$ that $w^{-1}\bn_t=\mat{-t}{-1}{1}{0}$, hence if $t\neq0$,
\begin{gather*}
\lm(w^{-1}\bn_t) = -t\\
\ab(w^{-1}\bn_t) = |t|.
\end{gather*}
Moreover, for $\lambda\in F^\times$, \[c_{w^{-1}}\left(\lm(w^{-1}\bn_t)^{-1}\right)\lambda=\lambda t,\] so that if $\varphi$ is a test function on $L$, then \[T_w\varphi=\int_{|t|<1}\varphi\left(-\frac{1}{t}\right)\,\frac{dt}{|t|}=\int_{|t|>1}\varphi(t)\,\frac{dt}{|t|}.\]
The common expression for the residual distribution in both the real and the complex case hence is \[T_w=\mathrm{1}_{\left\{|x|>1\right\}}\] where $\mathrm{1}_X$ the characteristic function of the set $X$.

\subsection{Principal series} Let us describe the dual of the Levi component. In the real case, $\widehat{M}_\R=\left\{\triv,\varepsilon\right\}$ where $\triv$ is the one-dimensional trivial representation and $\varepsilon(\pm I_2)=\pm 1$, while in the complex case, $\widehat{M}_\C=\left\{\varphi_n,n\in\mathbb{Z}\right\}\simeq\mathbb{Z}$ where $\varphi_n(z)=z^n$. In both situations, $A$ identifies with $\R_+^\times$, hence $\widehat{A}=\left\{\nu_t,t\in\R\right\}\simeq\R$ where $\nu_t(a)=a^{2i\pi t}$.

The action of $w$ by conjugation is given as follows:
\begin{gather*}
w.\varepsilon=\varepsilon\\
w.\triv=\triv\\
w.\varphi_n=\varphi_{-n}\\
w.\nu_t=\nu_{-t}.
\end{gather*}
Accordingly, the Weyl-fixed points in $\widehat{L}_r$ are
\begin{itemize}
\item $(\triv,\varphi_0)=(\triv,\triv)$ and $(\varepsilon,\varphi_0)=(\varepsilon,\triv)$ for $\GR$,
\item $(\triv,\varphi_0)=(\triv,\triv)$ for $\GC$.
\end{itemize}

Applying the theory of Knapp and Stein (see \cite{Knapp1}), one obtains that
\begin{itemize}
\item $\pi^{\triv,\triv}$ is irreducible on $\GR$,
\item $\pi^{\varepsilon,\triv}$ is reducible,
\item $\pi^{\triv,\triv}$ is irreducible on $\GC$.
\end{itemize}
The first and third points are in fact examples of the more general phenomenon discovered by B. Kostant \cite{Kostant} that principal series of the form $\pi^{\triv,\nu}$ are always irreducible. On the other hand, the second point can be made more precise: $\pi^{\varepsilon,\triv}$ splits into the direct sum of the so-called \textit{limits of the discrete series}, which is a special case of the Schmid equality (see \cite{Knapp1}). Finally, the third point should also be seen as a manifestation of a structural fact about $\GC$: N. Wallach proved in \cite{Wallach71} that complex groups, or more generally groups having exactly one conjugacy class of Cartan subgroups, only admit irreducible principal series representations. This situation corresponds to the case of a group $G$ having a Hausdorff tempered dual $\widehat{G}_r$, in which case the reduced $\Cs$-algebra is stably equivalent to $C_0(\widehat{G}_r)$.

\subsection{Detection of the reducibility parameters}

Let us now state the main result of this appendix. We denote by $\F_H$ the Fourier transform defined on characters of an abelian group $H$. The other notations are the same as above.

\begin{theorem}\label{distrireduc}
Let $(\sigma,\nu)\in\widehat{L}_r$ be Weyl-fixed and $T_w$ the distribution on $L$ associated to the residual part of the global interwining integral.
The principal series representation $\pi^{\sigma,\nu}$ is reducible if and only if $\F_L T_w(\sigma,\nu)=0$.
\end{theorem}

The rest of this section consists in a verification of this statement by determining the Fourier transform of the residual distribution associated to $\Iw$.

\subsubsection{\texorpdfstring{Fourier transform of $T_w$}{Fourier transform of Tw}}

Let us recall the expressions of Fourier transforms on the abelian groups $L_\R\simeq\R^\times\simeq\left\{\pm 1\right\}\times\R_+^*$ and $L_\C\simeq\C^\times\simeq\mathbb{S}^1\times\R_+^*$. With the notations introduced above for characters, one has, for a function $S$,
\begin{gather*}
\F_{\R^\times} S(\triv,\nu_t)=\int_{-\infty}^{+\infty}S(x)\overline{\triv\left(\frac{x}{|x|}\right)}|x|^{-2i\pi t}\,\frac{dx}{|x|} = \int_{-\infty}^{+\infty}S(x)|x|^{-2i\pi t}\,\frac{dx}{|x|}\\
\F_{\R^\times} S(\varepsilon,\nu_t)=\int_{-\infty}^{+\infty}S(x)\overline{\varepsilon\left(\frac{x}{|x|}\right)}|x|^{-2i\pi t}\,\frac{dx}{|x|} = \int_{-\infty}^{+\infty}S(x)\mathrm{sign}(x)|x|^{-2i\pi t}\,\frac{dx}{|x|}\\
\F_{\C^\times} S(\varphi_n,\nu_t)=\int_{\mathbb{S}^1}\int_0^{+\infty}S(\theta,|x|)e^{-in\theta}|x|^{-2i\pi t}\,\frac{dx}{|x|}\,d\theta\\
\end{gather*}
and those formulas extend to distributions by duality. It follows, in the case of $\GR$, that
\begin{itemize}
\item $\F_{\R^\times} T_w(\triv,\nu_t)=2.\int_1^{+\infty}|x^{-2i\pi t}|\,\dfrac{dx}{|x|}=2.\mathcal{F}_\R\left({\mathrm{1}_{\R_+}}\right)(t)=\dfrac{1}{i\pi t} + \delta_0(t)$
\item $\F_{\R^\times} T_w(\varepsilon,\nu_t)=\int_1^{+\infty}|x^{-2i\pi t}|\,\dfrac{dx}{|x|}-\int_{-\infty}^{-1}|x^{-2i\pi t}|\,\dfrac{dx}{|x|}=0$
\end{itemize}

while in the case of $\GC$, the same computation yields
\begin{itemize}
\item $\F_{\C^\times} T_w(\triv,\nu_t)=\dfrac{1}{i\pi t} + \delta_0(t).$
\end{itemize}

\subsubsection{Conclusion}
In view of the above discussion, it is immediate to check that the Fourier transform of the residual distribution vanishes exactly in the parameters for which the corresponding principal series splits into irreducibles. This observation suggests that this distribution really plays the same in the Hilbert module picture of principal series as the poles in the classical Knapp-Stein theory. In particular, the globality of the $\Cs$-algebraic point of view does not lead to the loss of information about individual representations. It seems plausible that Theorem~\ref{distrireduc} holds in general. If so, it would be interesting to find a direct proof, that is not involving the knowledge of the principal series resulting from the application of Knapp and Stein's theory. A more general goal would be to relate the objects constructed here to Harish-Chandra's Plancherel formula.

\bigskip
\begin{center}
\textbf{Acknowledgements}
\end{center}
The author thanks P. Julg and V. Lafforgue for many helpful discussions during the preparation of this article.

\bibliographystyle{amsalpha}
\bibliography{biblio}

\end{document}